\documentclass[12pt,A4,reqno]{amsart}
\usepackage{amsfonts}
\usepackage{mathrsfs}
\usepackage[T1]{fontenc}
\usepackage{amssymb}
\usepackage{amsmath,amscd}
\usepackage{color}
\usepackage{epsf}
\usepackage{graphicx}

\theoremstyle{plain}
\newtheorem{thm}{Theorem}[section]
\newtheorem{lem}[thm]{Lemma}

\newtheorem{cor}[thm]{Corollary}

\theoremstyle{definition}
\newtheorem{defi}[thm]{Definition}
\newtheorem{rem}[thm]{Remark}

\newcommand{\Q}{\mathbb Q}
\newcommand{\R}{\mathbb R}
\newcommand{\Z}{\mathbb Z}
\newcommand{\F}{\mathbb F}
\newcommand{\nn}{\vskip 0.2cm}
\newcommand{\n}{\vskip 0.1cm}

\begin{document}

\title [\ ] {On free $\Z_p$-torus actions in dimension two and three}

\author{Li Yu}
\address{Department of Mathematics and IMS, Nanjing University, Nanjing, 210093, P.R.China
 }
 \email{yuli@nju.edu.cn}
 \thanks{
 The author is partially supported by
 Natural Science Foundation of China (grant no.11371188) and the
 PAPD (priority academic program development) of Jiangsu higher education institutions.}

%\date{September 2, 2012}


%\subjclass[2000]{ }

\begin{abstract}
 We confirm the Halperin-Carlsson Conjecture for free $\Z_p$-torus
actions ($p$ is a prime) on $2$-dimensional finite CW-complexes
and free $\Z_2$-torus actions on compact $3$-manifolds.
 \end{abstract}

\maketitle

 \section{Introduction}

   Let $S^1=\{ z\in \mathbb{C} \, |\, |z|=1 \}$ be the circle group,
   and let $\Z_p$ be the abelian group $\Z\slash p\Z$ where $p$ is a prime.
   To avoid ambiguity, we use $\F_p$ to denote $\Z\slash p\Z$ as a field of coefficients.
  For any topological space $X$, let $H_i(X;\mathbb{F})$ and $H^i(X;\mathbb{F})$ denote the
  singular homology and singular cohomology groups of $X$ with coefficients $\mathbb{F}$, respectively. Define
    $$ b_i(X;\mathbb{F}) = \dim_{\mathbb{F}} H_i(X;\mathbb{F}) = \dim_{\mathbb{F}} H^i(X;\mathbb{F}),\
      \forall\,i\geq 0, \ \text{where}\ \F =\F_p \ \text{or}\ \mathbb{Q}. $$
   For any group $G$, let $H_i(G;\F_p)$ be the group homology of $G$ with $\F_p$-coefficient and let $b_i(G;\F_p) = \dim_{\F_p} H_i(G;\F_p)$.\nn
   \vskip .2cm

    \textbf{Halperin-Carlsson Conjecture (Toral Rank Conjecture):}\n
     If a finite dimensional space $X$ admits a free $(\Z_p)^r$-action ($p$ is
    a prime) or an almost free $(S^1)^r$-action, then $\sum^{\infty}_{i=0}
    b_i(X; \F_p)\geq 2^r$ or $\sum^{\infty}_{i=0} b_i(X; \Q)\geq 2^r$ respectively.\vskip .2cm

     The group $(\Z_p)^r$ is called a $\Z_p$-torus of rank $r$. For convenience, we define
     \[
          \mathrm{hrk}(X; \F) := \sum^{\infty}_{i=0} b_i(X; \F),\ \text{where}\ \F
           =\F_p \ \text{or}\ \mathbb{Q}.
        \]

       The above conjecture was proposed in the middle of 1980s by S.~Halperin
       in~\cite{Halperin85} for the torus case, and by G.~Carlsson
       in~\cite{Cal85} for the $\Z_p$-torus case.  \n

   In earlier time, this conjecture mainly took the form of
   whether the existence of a free $(\Z_p)^r$-action on a product of spheres
   $S^{n}\times \cdots \times S^{n_k}$ implies $r\leq k$.
   Many authors have studied this intriguing conjecture
   and contributed results with respect to different aspects
   (see~\cite{Adem87, AdemBrowd88, Adem98, Carl82, Corner57, Hanke09}).
   The reader is referred to~\cite{Adem04} and ~\cite{Allday93} for a survey of this subject.
    But the general case is still wide open.\n

   For general finite dimensional spaces, Halperin-Carlsson Conjecture was proved in~\cite{Puppe09} for
    $r\leq 3$ in the torus and $\Z_2$-torus
    cases and $r\leq 2$ in the $\Z_p$-torus ($p$ is odd prime) case.
     Some other evidences supporting this conjecture can be found in Cao-L\"u~\cite{LuZhi09},
     Ustinovskii~\cite{Uto09},
    Choi-Masuda-Oum~\cite{ChoiMasOum10}, F\'elix-Oprea-Tanr\'e~\cite{FOT08}, 
      Kamishima-Nakayama~\cite{KamNaka12}
     and Yu~\cite{Yu2012} in various settings.
       \n

  In this paper, we will mainly study the Halperin-Carlsson Conjecture for free
  cellular $(\Z_p)^r$-actions ($r\geq 1$) on a finite
  CW-complex $X$. In other words, $X$ can be thought of as
    a regular covering space over a finite
  CW-complex whose deck transformation group is $(\Z_p)^r$.
   In addition, the following conventions are adopted throughout this
 paper.\n
    \begin{itemize}
    \item  If not particularly indicated, any CW-complexes (or manifolds) and 
    their covering spaces
     in this paper are assumed to be path-connected.\n

    \item The rank $r$ of the group $(\Z_p)^r$ is at least $1$, i.e. $(\Z_p)^r$ is never trivial.\n

     \item For any set $S$, we use $|S|$ to denote the number of elements in
    $S$.
    \end{itemize}
    \n

  The main results of this paper are the following theorems.\n

  \begin{thm} \label{thm:Main1}
   If $(\Z_p)^r$ can act freely and cellularly on a 
    $2$-dimensional finite CW-complex $X$,
     then $\mathrm{hrk}(X; \F_p) \geq 2^r$.
      In particular, if $\mathrm{hrk}(X; \F_p) =  2^r$, then $r$ must be $1$ or $2$
         and, $X$ and the orbit space $K=X\slash (\Z_p)^r$ 
         must satisfy one of the following conditions:\n
       \begin{itemize}
         \item[(a)] $r=1$, $H_*(K;\F_p) \cong H_*(X; \F_p) \cong
         H_*(S^1;\F_p)$; \n

         \item[(b)] $r=1$, $p=2$, $H_*(K;\F_2) \cong H_*(\R P^2;\F_2)$, $H_*(X;\F_2) \cong
         H_*(S^2;\F_2)$; \n

         \item[(c)] $r=2$, $H_*(K;\F_p) \cong H_*(X; \F_p) \cong H_*(S^1\times
         S^1;\F_p)$.
       \end{itemize}
   \end{thm}

     \begin{thm}  \label{thm:Main2}
    Suppose $(\Z_p)^r$ acts freely and cellularly on a 
     finite CW-complex $X$. If the deficiency of the fundamental group
       of the orbit space $K=X\slash (\Z_p)^r$ is greater or equal to $1$, then we can conclude:
       \begin{itemize}
        \item $b_1(X;\F_p)\geq 2^{r-1}$ and
       $\mathrm{hrk}(X;\F_p) \geq 2^r$;\n
       
       \item if $\mathrm{hrk}(X;\F_p) = 2^r$, it is necessary that $b_1(K;\F_p) = r = 1$ or $2$.
       \end{itemize}
  \end{thm}

   \begin{thm} \label{thm:Main3}
   Suppose $(\Z_2)^r$ acts freely and cellularly on a 
     finite CW-complex $X$.
     If the deficiency of the fundamental group of the orbit space $K=X\slash (\Z_2)^r$
      is greater or equal to $0$, then we can conclude:
      \begin{itemize}
       \item $\mathrm{hrk}(X;\F_2) \geq 2^r$;\n
       \item if $\mathrm{hrk}(X;\F_2) = 2^r$, it is necesssy that $b_1(K;\F_2) = r \leq 3$.
       \end{itemize}
    \end{thm}

   \begin{rem}
    In the above theorems, the dimension of $X$ must be greater than $0$ since we assume
    that $X$ is path-connected and $r\geq 1$.
   \end{rem}
    It is well known that the fundamental group of any closed connected $3$-manifold 
    has deficiency greater than or equal to $0$. Indeed,  
     any closed connected $3$-manifold $Q$
    has Heegaard splitting (see~\cite[Chapter 5]{CoZie93}) which leads to 
    a finite presentation of the fundamental group of $Q$ with equal number of generators and
    relators. Then
     we have the following corollary of Theorem~\ref{thm:Main3}
      for free $(\Z_2)^r$-actions on closed $3$-manifolds.\n

    \begin{cor}\label{Cor:3-manifolds}
      If a closed connected $3$-manifold $M$ admits a free $(\Z_2)^r$-action,
      then $\mathrm{hrk}(M;\F_2) \geq 2^r$.
      In particular, if $\mathrm{hrk}(M;\F_2) = 2^r$, then $r\leq 3$ and,
      $M$ and the orbit space $Q=M\slash (\Z_2)^r$ must satisfy one of the following conditions:\n
          \begin{itemize}
              \item[(a)] $r=1$, $H_*(Q;\F_2)\cong H_*(\R P^3;\F_2)$, $H_*(M;\F_2)\cong
            H_*(S^3;\F_2)$;\n

             \item[(b)] $r=2$,  $H_*(Q;\F_2)\cong H_*(S^1\times \R P^2;\F_2)$,
             $H_*(M;\F_2)\cong H_*(S^1\times S^2;\F_2)$;\n

              \item[(c)] $r=3$, $H_*(Q;\F_2)\cong H_*(M;\F_2)\cong H_*(S^1\times S^1\times
             S^1;\F_2)$.
            \end{itemize}
    \end{cor}

     Theorem~\ref{thm:Main1} and
    Corollary~\ref{Cor:3-manifolds} motivate us to ask the following question.\nn

    \noindent \textbf{Question:} If there exists a free $(\Z_p)^r$-action ($r\geq 1$) on a closed
     connected manifold $M$ where $\mathrm{hrk}(M;\F_p) = 2^r$, then does  $H_*(M;\F_p)$
     necessarily agree with
       $H_*(S^{n_1}\times \cdots \times S^{n_r};\F_p)$ for some
       integer $n_1,\cdots, n_r$?\n

    All the examples known to the author so far
     give positive answer to
    the above question. Moreover,
    we can strengthen the question by replacing the ``closed connected manifold''
    by ``path-connected finite CW-complex'' and ask whether the same conclusion holds.\n

  The paper is organized as follows.
   In Section 2, we study finitely presented groups and their presentation 
   complexes. We prove that any finite presentation of a group can be
  transformed to another presentation of the group with the same number of generators and
  relators which satisfies some special property.
   In Section 3, we obtain a family of lower bounds of the rank of the first homology of
  a normal subgroup $N$ in a group $G$ with $G\slash N \cong (\Z_p)^r$.
  In Section~\ref{Sec3}, we use these lower bounds to prove Theorem~\ref{thm:Main1},
  Theorem~\ref{thm:Main2}, Theorem~\ref{thm:Main3} and Corollary~\ref{Cor:3-manifolds}.
  To make the paper more self-contained, we put two appendixes in the end. 
  In Appendix-1, we review some basic facts of group rings. In Appendix-2, we study
  the property of a family of integers $|\Omega^m_{p,r}|$ that are useful for 
  our arguments.  \\

 \section{Finitely presented groups and presentation
   complexes}\label{Sec2}

  Suppose $G$ be a finitely presentable group. Let
     $\mathcal{P} = \langle a_1\cdots, a_n \,| \, R_1,\cdots, R_m \rangle $ be
      a finite presentation of $G$. The integer $n-m$ is called the \emph{deficiency} of
       $\mathcal{P}$.
      When $m=n$, $\mathcal{P}$ is called a \emph{balanced
      presentation}. The \emph{deficiency} of $G$
      is the maximum over all its finite presentations, of the deficiency of each presentation.\nn

    Any finite presentation $\mathcal{P}$ canonically determines
      a $2$-dimensional CW-complex $K_{\mathcal{P}}$ called the
      \emph{presentation complex} of $\mathcal{P}$.
         \n
     \begin{itemize}
       \item $K_{\mathcal{P}}$ has a single vertex $q_0$,
       and one oriented $1$-cell $\gamma_j$ attached to $q_0$ for each generator $a_j$
    ($1\leq j \leq n$). So the $1$-skeleton of $K_{\mathcal{P}}$ is a bouquet of $n$ circles
     attached to $q_0$.\n

       \item $K_{\mathcal{P}}$ has one oriented $2$-cell $\beta_i$ for each 
       relator $R_i$ ($1\leq i \leq
       m$), where $\beta_i$ is attached to the $1$-skeleton of 
       $K_{\mathcal{P}}$ via a map defined by
         $R_i$.\n
     \end{itemize}

  The following are some well known facts. For any prime $p$,\nn
  \begin{itemize}
    \item $H_1(K_{\mathcal{P}};\F_p) \cong H_1(G;\F_p) \cong G\slash [G,G] G^p$, so 
      $b_1(K_{\mathcal{P}};\F_p) = b_1(G;\F_p)$.\nn
      
      \item $b_2(K_{\mathcal{P}};\F_p) \geq b_2(G;\F_p)$.
   \end{itemize}
      
 \begin{lem} \label{Lem:b1-b2}
  Suppose a group $G$ admits a finite presentation of deficiency $d$, i.e. 
  the deficiency of $G$ is greater or equal to $d$. Then $b_1(G;\F_p) - b_2(G;\F_p) \geq d $.
   \end{lem}
    \begin{proof}
  Let $ \mathcal{P}= \langle a_1,\cdots, a_n \, | \, R_1,\cdots, R_{n-d} \rangle$
    be a deficiency-$d$ presentation of $G$. 
      The Euler characteristic of the presentation complex $K_{\mathcal{P}}$ is:
     $$\chi(K_{\mathcal{P}}) = 1 - b_1(K_{\mathcal{P}};\F_p) + b_2(K_{\mathcal{P}};\F_p) 
     = 1 -n + (n-d) = 1-d.$$
    So $ b_1(G;\F_p) = b_1(K_{\mathcal{P}};\F_p) = 
         b_2(K_{\mathcal{P}};\F_p) + d $. Then since $b_2(K_{\mathcal{P}};\F_p) \geq b_2(G;\F_p)$,
         we obtain $b_1(G;\F_p) - b_2(G;\F_p) \geq d $.
  \end{proof}
 \nn

     Let
    $(C_*(K_{\mathcal{P}};\F_p), \partial^{\mathcal{P}}_*)$ be the
      cellular chain complex of $K_{\mathcal{P}}$ with $\F_p$-coefficients.
       \[
         0\longrightarrow C_2(K_{\mathcal{P}};\F_p)
         \overset{\partial^{\mathcal{P}}_2}{\longrightarrow} C_1(K_{\mathcal{P}};\F_p)
           \overset{\partial^{\mathcal{P}}_1}{\longrightarrow} C_0(K_{\mathcal{P}};\F_p) \longrightarrow 0
       \]
      \[ \text{Then}\ \,
      \
       C_1(K_{\mathcal{P}};\F_p) = \bigoplus^{n}_{j=1} \langle\gamma_j\rangle, \ \,
       C_2(K_{\mathcal{P}};\F_p) = \bigoplus^{m}_{i=1} \langle\beta_i\rangle,\
      \qquad
       \]
        where $\langle \gamma_j\rangle$ and $\langle\beta_i\rangle$
       are the subspaces of $C_1(K_{\mathcal{P}};\F_p)$ and $C_2(K_{\mathcal{P}};\F_p)$
       spanned by $\gamma_i$ and $\beta_j$, respectively.
       The map $\partial^{\mathcal{P}}_2$ can be represented by a matrix $A=(a_{ij})_{n\times
       m}$, $a_{ij}\in \F_p$, so that
   \[
    \partial^{\mathcal{P}}_2(\beta_1,\cdots, \beta_{m}) = (\gamma_1,\cdots,
    \gamma_{n})   \begin{pmatrix}
        a_{1,1} & \cdots & a_{1,m}    \\
        \vdots & \cdots & \vdots  \\
        a_{n, 1} & \cdots & a_{n, m}
   \end{pmatrix}
    \]
  It is clear that
   $ \dim_{\F_p}\mathrm{ker}(\partial^{\mathcal{P}}_2) = b_2(K_{\mathcal{P}};\F_p)$ and
   so $\mathrm{rank}_{\F_p}(A) = m-b_2(K_{\mathcal{P}};\F_p)$. In addition, by the Euler
    characteristic of $K_{\mathcal{P}}$, we have:
    \[ \chi(K_{\mathcal{P}}) = 1-n +m =1-b_1(K_{\mathcal{P}};\F_p) + b_2(K_{\mathcal{P}};\F_p) . \]
    So we get $b_2(K_{\mathcal{P}};\F_p) = b_1(K_{\mathcal{P}};\F_p) - n
    +m$, and so
    \[
    \ \  \mathrm{rank}_{\F_p}(A) = n -  b_1(K_{\mathcal{P}};\F_p) = n - b_1(G;\F_p).
     \]

    It is easy to see that we can use two types of
   elementary transformations to turn the matrix $A$ into its \emph{Smith normal
   form} (we do not require the nonzero entries in the Smith normal form
    to be $1 \in \F_p$).
    \nn
    \begin{enumerate}
      \item[Type 1:] multiply one row (or column) of $A$ by a nonzero element of $\F_p$ and
      then add it to another row (or column),\nn

      \item[Type 2:] switch two rows (or two columns) of $A$.\nn
    \end{enumerate}
    In other words, there exists a sequence of elementary transformation
   matrices $P_1,\cdots, P_s$ and $Q_1, \cdots, Q_t$ over $\F_p$ so that
   \begin{equation}\label{Equ:Matrix}
      P_s\cdots P_1 A Q_1\cdots Q_t = \begin{pmatrix}
         D & \mathbf{0}    \\
        \mathbf{0} & \mathbf{0}
     \end{pmatrix}
   \end{equation}
   where $D$ is a diagonal matrix of size
   $n -  b_1(G;\F_p)$ whose diagonals are nonzero elements of
   $\F_p$, and $P_1,\cdots, P_s, Q_1,\cdots, Q_t$ are matrices with one of the following
   forms:
   \[
      T_{ij}(q) = \bordermatrix{%
         &  & &  i & & j & & \cr
         & 1 &  & & & &  &  \cr
         & & \ddots & & & &  &  \cr
       i & &  &  1 & & &  &  \cr
         & &  & &  \ddots & &  &  \cr
       j & &  &  q & &  1 &  &  \cr
         & &  & & & & \ddots  &  \cr
         & &  & & & &  & 1
         },\  q\in \F_p;  \]
   \[ \mathrm{or} \ \;  S_{ij} = \bordermatrix{%
         &  & &  i & & j & & \cr
         & 1 &  & & & &  &  \cr
         & & \ddots & & & &  &  \cr
       i & &  &  0 & & 1 &  &  \cr
         & &  & &  \ddots & &  &  \cr
       j & &  &  1 & &  0 &  &  \cr
         & &  & & & & \ddots  &  \cr
         & &  & & & &  & 1
         }. \qquad \qquad
   \]\nn
   Notice that $T^{-1}_{ij}(q) = T_{ij}(-q)$ and $S^{-1}_{ij} = S_{ij}$.\nn

   \begin{rem}
    In the process of turning $A$ into its Smith normal form, we do not need
   to multiply one row (or column) by a nonzero scalar in $\F_p$
   since we do not require the nonzero entries in the Smith normal form
   to be $1 \in \F_p$.
   \end{rem}

   Next, we wish to transform the presentation $\mathcal{P}$ so that
   the boundary map $\partial^{\mathcal{P}}_2$ in $(C_*(K_{\mathcal{P}};\F_p), \partial^{\mathcal{P}}_*)$ has a simpler form. 
   The transformations we are going to use are called Nielson transformations.
     \nn
  \begin{defi}[Nielsen transformation]
    The following transformations of a set $\{ W_1,\cdots, W_m\}$ of
    freely reduced words in the free group $F = F[a_1,\cdots, a_k]$
    are called \emph{elementary Nielsen transformations of words}:
     \begin{itemize}
       \item[(i)]  permuting the $W_i$ and taking inverses of some of them;
       \item[(ii)] leaving fixed all $W_i$, $i\neq j$, and replacing
       $W_j$ by the freely reduced form of any one of the following:
        $W_jW^q_i$, $W_i W_j W^{-1}_i$ and $W^{-1}_i W_j W_i$ where $q\in \Z$ and $1\leq j < i \leq
        m$.
     \end{itemize}

     Let $\mathcal{P}=\langle a_1,\cdots, a_n \, |\, R_1,\cdots, R_m \rangle$
    be a presentation for a group $G$. 
    \begin{itemize}
      \item If an elementary
  Nielsen transformation is applied to the set of
  relators $\{ R_1,\cdots, R_m \}$, it produces a set of relators $\{ R^*_1,\cdots, R^*_m \}$
  for a new presentation
  $\mathcal{P}^* = \langle a_1,\cdots, a_{n} \, |\, R^*_1,\cdots, R^*_m \rangle$ of $G$.\n
    \item If an elementary Nielsen transformation is applied to the set
  $\{ a_1,\cdots, a_n  \}$ of generators, it produces an alternative
  set of free generators $\{ a'_1,\cdots, a'_n \}$ for the free group $F[a_1,\cdots,
  a_n]$. The relators $ R_1,\cdots, R_m $, when written as
  reduced words in the new generators, determine relators $R'_1,\cdots, R'_m$
  for a new presentation $\mathcal{P}'=\langle a'_1,\cdots, a'_n \, |\, R'_1,\cdots, R'_m \rangle$
  of $G$.
  \end{itemize}
  
    The two transformations $\mathcal{P} \rightarrow \mathcal{P}^*$
  and $\mathcal{P} \rightarrow \mathcal{P}'$ are called \emph{elementary
  Nielsen transformations of presentations}. Finite compositions of
  such are referred to as \emph{Nielsen transformations of presentations}.
  Two presentations are called \emph{Nielsen equivalent} if they can be related by
  Nielsen transformations.
 \end{defi}
 
  It is clear that Nielsen
  transformations of a presentation $\mathcal{P}$ do not change the number of generators and
  relators of $\mathcal{P}$. Moreover, \cite[Proposition 2]{DyerSier73} shows that Nielsen
  transformations do not change the simple
  homotopy type of the presentation complex
  $K_{\mathcal{P}}$ of $\mathcal{P}$.\nn

\begin{lem} \label{Lemma:Presentation}
  For any finite presentation $\mathcal{P}=\langle a_1,\cdots, a_{n} \, |\, R_1,\cdots,
  R_{m}\rangle$ of a group $G$, there exists another presentation
  $\widehat{\mathcal{P}}=\langle \widehat{a}_1,\cdots, \widehat{a}_{n} \, | \,
    \widehat{R}_1 ,\cdots, \widehat{R}_{m} \rangle$ of $G$
     so that:
     \begin{enumerate}
       \item[(i)] $\widehat{\mathcal{P}}$ is Nielsen equivalent to
       $\mathcal{P}$.\n

       \item[(ii)] the boundary map $\partial^{\widehat{\mathcal{P}}}_2$ in the cellular chain complex
       $(C_*(K_{\widehat{\mathcal{P}}};\F_p), \partial^{\widehat{\mathcal{P}}}_*)$ of
     $K_{\widehat{\mathcal{P}}}$ is represented by an $n\times m$
     matrix
  $\begin{pmatrix}
         D & \mathbf{0}    \\
        \mathbf{0} & \mathbf{0}
     \end{pmatrix}$ where $D$ is a diagonal square matrix of size
   $n -  b_1(G;\F_p)$ whose diagonals are nonzero elements of
   $\F_p$.
     \end{enumerate}
\end{lem}
\begin{proof}
     For any sequence of freely reduced words $(\omega_1,\cdots, \omega_{m})$
     in  the free group $F[a_1,\cdots, a_{n}]$ generated by $a_1,\cdots, a_{n}$,
      let $N (\omega_1,\cdots, \omega_{m})$
     be the normal subgroup of $F[a_1,\cdots, a_{n}]$ generated by $\omega_1,\cdots,
     \omega_{m}$. For any matrix $T_{ij}(q)$, define
     \begin{equation}\label{Equ:T_ij(q)}
     (\omega_1,\cdots, \omega_{m}) T_{ij}(q) = (\omega_1,\cdots, \omega_{i-1},\omega_i\omega^q_j, \omega_{i+1},
     \cdots, \omega_{m}), \qquad \quad
     \end{equation}
     where $q\in \Z_p$ here is thought of as an integer in $\{ 0,\cdots, p-1
     \}$. For any $S_{ij}$, define
       \begin{equation}\label{Equ:S_ij}
         (\omega_1,\cdots, \omega_{m}) S_{ij} = (\omega_1,\cdots, \omega_{i-1}, \omega_j, \omega_{i+1},
         \cdots, \omega_{j-1}, \omega_{i}, \omega_{j+1}, \cdots, \omega_{m}).
     \end{equation}
     \n
        Suppose $\partial^{\mathcal{P}}_2$ in $(C_*(K_{\mathcal{P}};\F_p), \partial^{\mathcal{P}}_*)$
      is represented by a matrix $A$ with the form
      in~\eqref{Equ:Matrix}. Let $(\widetilde{R}_1,\cdots, \widetilde{R}_{m}) =
     (R_1,\cdots, R_{m}) Q_1 \cdots Q_t$. So the following presentation
    \begin{equation}\label{Equ:G2}
     \widetilde{\mathcal{P}} = \langle a_1,\cdots, a_{n}\, | \, \widetilde{R}_1,\cdots, \widetilde{R}_{m} \rangle
    \end{equation}
    is Nielsen equivalent to $\mathcal{P}$ by~\eqref{Equ:T_ij(q)}~\eqref{Equ:S_ij}.
    Let $K_{\widetilde{\mathcal{P}}}$ be the presentation complex of
  $\widetilde{\mathcal{P}}$. Then in the chain complex
  $(C_*(K_{\widetilde{\mathcal{P}}};\F_p),
  \partial^{\widetilde{\mathcal{P}}}_*)$, the boundary map
  $\partial^{\widetilde{\mathcal{P}}}_2$ is represented by $A Q_1 \cdots Q_t$.
 Moreover, let
 $$ \qquad (a'_1,\cdots, a'_{n}) = (a_1,\cdots, a_{n})
 T^{-1}_{ij}(q)= (a_1,\cdots, a_{n}) T_{ij}(-q),$$
 $$(a''_1,\cdots, a''_{n}) = (a_1,\cdots, a_{n}) S^{-1}_{ij} = (a_1,\cdots, a_{n}) S_{ij}. $$
   \[ \text{So by definition,}\ \
      a'_l = \left\{
         \begin{array}{ll}
           a_l , & \hbox{\text{$l\neq i$};} \\
           a_ia^{-q}_j, & \hbox{\text{$l=i$}.}
         \end{array}
       \right.\ \
    a''_l = \left\{
         \begin{array}{ll}
           a_l , & \hbox{\text{$l\neq i,j$};} \\
           a_j , & \hbox{\text{$l=i$};} \\
           a_i, & \hbox{\text{$l=j$}.}
         \end{array}
       \right.
   \qquad \qquad \ \]

  Replacing the generators $a_1,\cdots, a_{n}$
 by $a'_1,\cdots, a'_{n}$ and $a''_1,\cdots, a''_{n}$ respectively and
 rewrite the relators accordingly, we will get two new presentations of
 $G$:
  $$
   \widetilde{\mathcal{P}}'= \langle a'_1,\cdots, a'_{n}\, | \,
   \widetilde{R}'_1,\cdots, \widetilde{R}'_{m} \rangle,
   \
    \widetilde{\mathcal{P}}''=\langle a''_1,\cdots, a''_{n}\, | \,
    \widetilde{R}''_1,\cdots, \widetilde{R}''_{m} \rangle
   $$
  which, by definition, are both Nielsen equivalent to $\widetilde{P}$.
  Let $K_{\widetilde{\mathcal{P}}'}$ and $K_{\widetilde{\mathcal{P}}''}$ be the presentation complex
   of $ \widetilde{\mathcal{P}}'$ and $ \widetilde{\mathcal{P}}''$
   respectively. Then the boundary map
  $\partial^{\widetilde{\mathcal{P}}'}_2$ and  $\partial^{\widetilde{\mathcal{P}}''}_2$
  in the cellular chain complexes $C_*(K_{\widetilde{\mathcal{P}}'};\F_p)$ and $C_*(K_{\widetilde{\mathcal{P}}''};\F_p)$
  are represented by $T_{ij}(q) A Q_1 \cdots Q_t$ and $S_{ij} A Q_1 \cdots Q_t$
  respectively.  \nn

  By iterating the above process of transforming the
  presentation of $G$ according to the matrices $P_1,\cdots, P_s$, we will eventually get a new presentation of
  $G$
  \begin{equation}\label{Equ:G4}
   \widehat{\mathcal{P}} = \langle \widehat{a}_1,\cdots, \widehat{a}_{n} \, | \, \widehat{R}_1 ,\cdots, \widehat{R}_{m} \rangle
 \end{equation}
  whose generators
 $(\widehat{a}_1,\cdots, \widehat{a}_{n})= (a_1,\cdots, a_{n}) P^{-1}_1\cdots P^{-1}_s$.
  Let $K_{\widehat{\mathcal{P}}}$ be the presentation complex of
  $\widehat{\mathcal{P}}$. The map $\partial^{\widehat{\mathcal{P}}}_2$ in the cellular chain complex
  $(C_*(K_{\widehat{\mathcal{P}}};\F_p), \partial^{\widehat{\mathcal{P}}}_*)$ of $K_{\widehat{\mathcal{P}}}$
   is represented by $P_s\cdots P_1 A Q_1\cdots Q_t$ (see~\eqref{Equ:Matrix}).
  So the lemma is proved.
 \end{proof}
 \nn
   
  \section{Lower bounds of the first homology group
  of finite index normal subgroups}

   In this section, we first prove the following theorem which is the
  driving force behind the proofs of all other results of this paper.\nn

  \begin{thm}\label{thm:Homology-Estimate-H}
     Suppose a group $G$ admits finite presentation of deficiency $d$, i.e. the deficiency of $G$ is
     at least $d$.
    Then for any prime $p$ and any finite index normal subgroup $N$ of $G$ with $G\slash N\cong H$,
     \begin{equation*}
       b_1(N;\F_p) \geq 1 + b_1(G;\F_p)  \lambda^k_p(H) + d \, \sum^{k-1}_{j=0}
           \lambda^j_p(H) - |H| , \ \, \forall\,k\geq 0,
     \end{equation*}
      where $\lambda^k_p(H) = \dim_{\F_p}\Delta^k_{\F_p}(H)\slash
      \Delta^{k+1}_{\F_p}(H)$, $\Delta_{\F_p}(H)$ is the augmentation ideal of
      the group ring $\F_p[H]$. Note that $b_1(G;\F_p) \geq d$.
  \end{thm}

   Let $ \mathcal{P}= \langle a_1,\cdots, a_n \, | \, R_1,\cdots, R_{n-d} \rangle$
    be a deficiency-$d$ presentation of $G$.  Let
   $K$ be the presentation complex of
   $\mathcal{P}$. Then $b_1(K;\F_p)= b_1(G;\F_p)$.\nn

    The Euler characteristic of $K$ is
     $$\chi(K) = 1 - b_1(K;\F_p) + b_2(K;\F_p) = 1 -n + (n-d) =
    1-d.$$
    $$ \Longrightarrow\ \ b_1(K;\F_p) = b_2(K;\F_p) + d \geq d,\ \text{so}\ b_1(G;\F_p)\geq d. $$
    By Lemma~\ref{Lemma:Presentation}, we can assume that
   in the cellular chain complex of $K$,
     \begin{equation} \label{Equ:ChainComplex-K}
         0\longrightarrow C_2(K;\F_p) \overset{\partial^K_2}{\longrightarrow}
             C_1(K;\F_p)
           \overset{\partial^K_1}{\longrightarrow} C_0(K;\F_p) \longrightarrow
           0,
      \end{equation}
      the boundary map $\partial^K_2$ is represented by
    an $n \times n-d$ matrix $ U = \begin{pmatrix}
         D & \mathbf{0}    \\
        \mathbf{0} & \mathbf{0}
     \end{pmatrix}$
    where $D$ is a diagonal matrix of size
   $n - b_1(G;\F_p)$ whose diagonals are nonzero elements of
   $\F_p$.
   So all the entries of $U$ are $0$ except $U_{ii}$, $1\leq i \leq n-b_1(G;\F_p)$.
 \nn

  Let $X$ be a regular covering space over $K$ corresponding to the normal subgroup 
  $N$ of $G \cong \pi_1(K)$. So we have
      $$  b_1(N;\F_p) = b_1(X;\F_p).$$
    Let $\xi: X \rightarrow K$ be the covering map.
     Let $q_0$ be the single $0$-cell in $K$ and let the set of oriented $1$-cells
   and $2$-cells of $K$ be
   $\{ \gamma_1,\cdots, \gamma_n \}$ and $\{ \beta_1,\cdots, \beta_{n-d} \}$.
    Then $X$ has a natural $H$-invariant cell
     structure induced from $K$ by $\xi$. We choose \n
     \begin{itemize}
       \item a point $x_0 \in \xi^{-1}(q_0)$;\n

       \item an oriented $1$-cell $\widetilde{\gamma}_j$ of $X$
     attached to $x_0$ so that $\xi(\widetilde{\gamma}_j) = \gamma_j$ ($1\leq j \leq
     n$);\n

       \item an oriented
     $2$-cell $\widetilde{\beta}_i$ of $X$ attached to $x_0$
   so that $\xi(\widetilde{\beta}_i) = \beta_i$  ($1\leq i \leq n-d$).\nn
     \end{itemize}
    Then the sets of $0$-cells, $1$-cells and $2$-cells in $X$
    are\n
   \begin{itemize}
     \item $0$-cells : $\{ h\cdot x_0 \,|\, h\in H \}$ denoted by $Hx_0$;\n

     \item $1$-cells : $\{ h\cdot \widetilde{\gamma}_j \,|\, h\in H, 1\leq j \leq n
     \}$ denoted by $H\widetilde{\gamma}_j$;\n
     \item $2$-cells : $\{ h\cdot \widetilde{\beta}_i \,|\, h\in H, 1\leq i \leq n-d
     \}$ denoted by $H\widetilde{\beta}_i$.\nn
   \end{itemize}

       We label $x_0$, $\widetilde{\gamma}_j$ and
      $\widetilde{\beta}_i$ by the identity $e\in H$, and label
    $h\cdot x_0$, $h\cdot \widetilde{\gamma}_j$ and $h\cdot \widetilde{\beta}_i$ by $h$ for
    any $h\in H$. Then every cell in $X$ is labeled by a unique
     element of $H$.
   \nn

     Let $(C_*(X; \F_p), \partial^X_*)$ be the
      cellular chain complex of $X$ with $\F_p$-coefficients.
      \begin{equation} \label{Equ:ChainComplex-X}
         0\longrightarrow C_2(X; \F_p) \overset{\partial^X_2}{\longrightarrow} C_1(X; \F_p)
           \overset{\partial^X_1}{\longrightarrow} C_0(X; \F_p) \longrightarrow 0
      \end{equation}
      \[ C_1(X; \F_p) = \bigoplus^n_{j=1} \langle H\widetilde{\gamma}_j\rangle,\ \,
       C_2(X; \F_p) = \bigoplus^{n-d}_{i=1} \langle H\widetilde{\beta}_i\rangle  \qquad \quad \]
       where $\langle H\widetilde{\gamma}_j\rangle$ and $ \langle H\widetilde{\beta}_i\rangle$
       are the free submodules of $C_1(X; \F_p)$ and $C_2(X; \F_p)$
       generated by the set $H\widetilde{\gamma}_j$ and
       $H\widetilde{\beta}_i$ over $\F_p$, respectively. Clearly,
        $$ \langle H\widetilde{\gamma}_j\rangle \cong
         \langle H\widetilde{\beta}_i\rangle \cong \F_p[H],$$
       The reader is referred to Appendix-1 for the basic facts of the group ring $\F_p[H]$. \nn

    Next, let us see what $\partial^X_2$ looks like with respect to the above $H$-invariant
     cell structure of $X$. Suppose \begin{equation}\label{Equ:Bound-1}
       \partial^X_2 \widetilde{\beta}_i = \sum^n_{j=1} \sum^{d_{ij}}_{s=1} h^{j}_{i,s}\cdot
       \widetilde{\gamma}_j,\ 1\leq i \leq n-d,\,  h^j_{i,s}\in H.
       \end{equation}
       Then since the cell structure of $X$ is $H$-invariant, we have
        \begin{equation}\label{Equ:Bound-2}
       \partial^X_2 (g\cdot \widetilde{\beta}_i) = \sum^n_{j=1} \sum^{d_{ij}}_{s=1} 
        (gh^{j}_{i,s})\cdot
       \widetilde{\gamma}_j,\  \forall g\in H, 1\leq i \leq n-d.
       \end{equation}
       
   Next, we fix an order of all the elements of $H$ and order
      the elements of $H\widetilde{\gamma}_j$ and $H\widetilde{\beta}_i$ accordingly. 
      Then we get an ordered basis
       of $C_1(X; \F_p)$ and $C_2(X; \F_p)$, denoted by
       $\{ H\widetilde{\gamma}_1,\cdots, H\widetilde{\gamma}_n
           \}$ and  $\{ H\widetilde{\beta}_1,\cdots,
          H\widetilde{\beta}_{n-d} \} $, respectively. Let
        \begin{equation} \label{Equ:B-def}
      \partial^X_2 \Big(
        H\widetilde{\beta}_1,
        \cdots,
         H\widetilde{\beta}_{n-d}
         \Big) =
     \Big(
      H\widetilde{\gamma}_1,
       \cdots,
       H\widetilde{\gamma}_n
      \Big) \begin{pmatrix}
        \mathbf{B}_{11} & \cdots & \mathbf{B}_{1,n-d}    \\
         \vdots & \cdots & \vdots \\
        \mathbf{B}_{n,1} & \cdots & \mathbf{B}_{n,n-d}
   \end{pmatrix} 
  \end{equation}
   where each $\mathbf{B}_{ij}$ ($1\leq i \leq n$, $1\leq j \leq n-d$) is a $|H|\times |H|$ matrix which tells us the
   weight of each $1$-cell $h\cdot\widetilde{\gamma}_j$ in
   the boundary of any $2$-cell $g\cdot\widetilde{\beta}_i$.
   Let
   \[
     \mathbf{B}=    \begin{pmatrix}
        \mathbf{B}_{11} & \cdots & \mathbf{B}_{1,n-d}    \\
         \vdots & \cdots & \vdots \\
        \mathbf{B}_{n,1} & \cdots & \mathbf{B}_{n,n-d}
   \end{pmatrix}.
   \]

   The rows and columns of each $\mathbf{B}_{ij}$ are both indexed by elements
  of $H$ with the given order. Let
  $\mathbf{B}_{ij}(g,h)\in \F_p$ denote the entry of $\mathbf{B}_{ij}$ with row index
  $g\in H$ and column index $h\in H$, and let $\mathbf{B}_{ij}(h)$ be the column
  vector of $\mathbf{B}_{ij}$ indexed by $h\in H$. In the following we identify $\mathbf{B}_{ij}(h)$ with the
  element $\sum_{g\in H} \mathbf{B}_{ij}(g,h) \delta_g \in \F_p[H]$.
  Then according to~\eqref{Equ:Bound-2}, we have
  \begin{equation}\label{Equ:Property-B}
    \mathbf{B}_{ij}(hh') = \delta_h* \mathbf{B}_{ij}(h'), \ \forall\, h, h'\in H.
  \end{equation}
   where $\ast$ is the product in the group ring $\F_p[H]$ 
  (see~\eqref{equ:ast-2}). In particular, we have
    \begin{equation} \label{Equ:Property-B2}
       \mathbf{B}_{ij}(h) = \delta_{h}* \mathbf{B}_{ij}(e), \ \forall\,h\in H.
    \end{equation}
   This means that all column vectors of $\mathbf{B}_{ij}$ are determined by
   the column $\mathbf{B}_{ij}(e)$. We will see that this property of $\mathbf{B}_{ij}$
   is the essential reason why we can derive the lower bound of
   $b_1(N;\F_p)$ in Theorem~\ref{thm:Homology-Estimate-H}.

  \begin{lem}\label{Lem:v-A}
      If we consider any $v\in\F_p[H]$ as a column vector, $\mathbf{B}_{ij}v =
        v*\mathbf{B}_{ij}(e)$.
\end{lem}
  \begin{proof}
   Suppose $v=\sum_{h\in H} l_h \delta_h$ where $l_h\in \F_p$. Then by~\eqref{Equ:Property-B2},
    we have
   $$ \mathbf{B}_{ij}v = \sum_{h\in H} l_h \mathbf{B}_{ij}(h) =
         \sum_{h\in H} l_h \big(\delta_h *\mathbf{B}_{ij}(e)\big) =
           \Big(\sum_{h\in H} l_h \delta_h \Big) *\mathbf{B}_{ij}(e)
     =v*\mathbf{B}_{ij}(e).$$
 \end{proof}

  In the rest of the paper, we do not distinguish the $|H|\times |H|$ matrix $\mathbf{B}_{ij}$ and
 the linear transformation on $\F_p[H]$ determined by $\mathbf{B}_{ij}$.\n
 
 \begin{lem} \label{Lem:A-Send-general}
   For any $k\geq 0$, $\mathbf{B}_{ij} : \F_p[H] \rightarrow \F_p[H]$ sends
        $\Delta^k_{\F_p}(H)$ into $\Delta^k_{\F_p}(H)$. Moreover, if
       $\mathbf{B}_{ij}(e) \in \Delta_{\F_p}(H) $, then 
       $\mathbf{B}_{ij}$ maps $\Delta^k_{\F_p}(H)$ into
      $\Delta^{k+1}_{\F_p}(H)$.\nn
 \end{lem}
 \begin{proof}
    Note that $\Delta^k_{\F_p}(H)$ ($k\geq 1$) is linearly spanned by the following set over $\F_p$
       \begin{equation} \label{Equ:Set-Phi}
       \{ (-\delta_e  + \delta_{g_1})*\cdots  *
       (-\delta_e  + \delta_{g_k}) \in \F_p[H],  \ \text{where} \ g_1,\cdots, g_k \in H\}
    \end{equation}
   For brevity, we define
    $v_{(g_1,\cdots, g_k)} := (-\delta_e + \delta_{g_1})*\cdots *
       (-\delta_e  + \delta_{g_k})$ for any $k\geq 1$.
      In particular, $v_{(h)} = -\delta_e + \delta_h$ for
    any $h\in H$.\n

   Suppose $\mathbf{B}_{ij}(e)= \sum_{h\in H} l_h \delta_h$, $l_h\in \F_p$. Then we have
     \begin{align*}
      \mathbf{B}_{ij}(e) &=  \big(\sum_{h\in H} l_h\big) \delta_e + \sum_{h\in H} (- l_h \delta_e +
        l_h \delta_h) \\
        & = \big(\sum_{h\in H} l_h\big) \delta_e + \sum_{h\in H} l_h (-\delta_e + \delta_h)
       =  \big(\sum_{h\in H} l_h\big) \delta_e + \sum_{h\in H} l_h v_{(h)}.
     \end{align*}

   So for any tuple $(g_1,\cdots, g_k)$ of elements in $H$, we have
   \begin{align*}
      \mathbf{B}_{ij}   v_{(g_1, \cdots , g_k)} &= v_{(g_1, \cdots ,g_k)} *\mathbf{B}_{ij}(e)
         = v_{(g_1 , \cdots , g_k)} * \Big(
      \big(\sum_{h\in H} l_h\big) \delta_e + \sum_{h\in H} l_h v_{( h )}  \Big)\\
       &= \big(\sum_{h\in H} l_h\big) v_{(g_1, \cdots , g_k)}  + \sum_{h\in H} l_h
  v_{(g_1, \cdots , g_k, h)}.
   \end{align*}
   Since $v_{(g_1, \cdots , g_k)} \in \Delta^k_{\F_p}(H)$,
   $ v_{(g_1 , \cdots , g_k, h)} \in \Delta^{k+1}_{\F_p}(H) \subset
  \Delta^k_{\F_p}(H)$, $\mathbf{B}_{ij}$ preserves $\Delta^k_{\F_p}(H)$. \nn

  If $\mathbf{B}_{ij}(e) \in \Delta_{\F_p}(H) $, i.e. $\sum_{h\in H} l_h=0$, we have
    $$ \mathbf{B}_{ij} v_{(g_1 , \cdots , g_k)}  =
      \sum_{h\in H} l_h v_{(g_1 , \cdots , g_k, h)} \in \Delta^{k+1}_{\F_p}(H).$$
   So $\mathbf{B}_{ij}$ maps $\Delta^k_{\F_p}(H)$ into
      $\Delta^{k+1}_{\F_p}(H)$ in this case. The lemma is proved.
 \end{proof}
 \nn
 
 \begin{proof}[\textbf{Proof of Theorem~\ref{thm:Homology-Estimate-H}}]\ \n
 
  To estimate $H_1(N;\F_p) = H_1(X;\F_p)$, we need to understand the kernal of
  $\partial^X_2 : C_2(X; \F_p)\rightarrow C_1(X; \F_p)$, which
    is given by the block matrix
    $\mathbf{B} = (\mathbf{B}_{ij})$
     \begin{equation}\label{Equ:B-map}
       \mathbf{B} :  \overset{n-d}{\overbrace{\F_p[H]\oplus \cdots \oplus \F_p[H]}} \, \longrightarrow
      \,  \overset{n}{\overbrace{\F_p[H]\oplus \cdots \oplus
      \F_p[H]}},
     \end{equation}
        where $\mathbf{B}_{ij}$ define a linear map from the $i$-th
        copy of $\F_p[H]$ on the left side to the $j$-th copy of $\F_p[H]$ on
        the right side of~\eqref{Equ:B-map}.\n

   By comparing $\partial^X_2$ with
    $\partial^K_2 : C_2(K;\F_p) \rightarrow C_1(K;\F_p)$ in~\eqref{Equ:ChainComplex-K},
     we observe that 
    \begin{equation} \label{Equ:B_ij_Balance}
      \qquad \quad
     \mathbf{B}_{ij}(e) \in \Delta_{\F_p}(H) \Longleftrightarrow U_{ij}=0.
    \end{equation}        
        So the form of $U$ (see~\eqref{Equ:ChainComplex-K}) implies that
        \begin{equation} \label{Equ:B_ij-U}
           \mathbf{B}_{ij}(e) \in \Delta_{\F_p}(H),\ 
              n-b_1(G;\F_p)+1 \leq i \leq n, 1\leq  j \leq n-d.
        \end{equation}    
            We claim that: for all $k\geq 0$,
    \begin{equation} \label{Equ:B-map-2}
      \overset{n-d}{\overbrace{\Delta^k_{\F_p}(H) \oplus \cdots \oplus
    \Delta^k_{\F_p}(H)}}
    \, \overset{\mathbf{B}}{\longrightarrow} \,
     \overset{n-b_1(G;\F_p)}{\overbrace{\Delta^k_{\F_p}(H)\oplus \cdots \oplus \Delta^k_{\F_p}(H)}}
       \oplus \overset{b_1(G;\F_p)}{\overbrace{\Delta^{k+1}_{\F_p}(H) \oplus \cdots \oplus
       \Delta^{k+1}_{\F_p}(H)}}.
      \end{equation}
  \n
  Indeed, for any $v_1,\cdots, v_{n-d} \in \Delta^k_{\F_p}(H) $, we have
    \[
     \mathbf{B} \begin{pmatrix}
         v_1 \\ \vdots \\ v_{n-d}
              \end{pmatrix} = \begin{pmatrix} \sum^{n-d}_{j=1} 
     \mathbf{B}_{1j}v_j \\ \vdots \\ \sum^{n-d}_{j=1}  \mathbf{B}_{nj} v_j
     \end{pmatrix}.
    \]
   By Lemma~\ref{Lem:A-Send-general}, each $\mathbf{B}_{ij}$
    maps $\Delta^k_{\F_p}(H)$ into $\Delta^k_{\F_p}(H)$ and,
    in particular when $\mathbf{B}_{ij}(e)\in \Delta_{\F_p}(H)$,
    $\mathbf{B}_{ij}$ maps $\Delta^k_{\F_p}(H)$ into $\Delta^{k+1}_{\F_p}(H)$.
     So by~\eqref{Equ:B_ij-U}, we have
    $$ \mathbf{B}_{ij} v_j \in  \Delta^{k+1}_{\F_p}(H), \ 
      n-b_1(G;\F_p)+1 \leq i \leq n, 1\leq  j \leq n-d.$$
    $$ \text{So}\ \ \sum^{n-d}_{j=1}  \mathbf{B}_{ij}v_j \in \Delta^{k+1}_{\F_p}(H), \
     n-b_1(G;\F_p)+1 \leq  i \leq n, \ \text{which prove the claim}. $$\nn

   From~\eqref{Equ:B-map-2}, we obtain the following inequality for each $k\geq 0$.
  \begin{align} \label{Equ:b2-XN}
    b_2 (X;\F_p)  &= \dim_{\F_p} \mathrm{ker}(\mathbf{B}) \notag\\
     &\geq (b_1(G;\F_p)-d) \dim_{\F_p}\Delta^k_{\F_p}(H) - b_1(G;\F_p)  
     \dim_{\F_p}\Delta^{k+1}_{\F_p}(H). \notag\\
      &\geq b_1(G;\F_p) \big(\dim_{\F_p}\Delta^k_{\F_p}(H) -  \dim_{\F_p}\Delta^{k+1}_{\F_p}(H)\big)
           - d \cdot\dim_{\F_p}\Delta^k_{\F_p}(H) \notag\\
    & \overset{\eqref{Equ:lambda-k-H-2}}{=}
      b_1(G;\F_p)\lambda^k_p(H)  -  d \cdot \Big(|H|- \sum_{0\leq j \leq k-1} \lambda^j_p(H) \Big)
  \end{align}
 
  In addition, the Euler characteristic of $X$ is
  \[
    \chi(X) = 1- b_1 (X;\F_p) + b_2 (X;\F_p) = |H|\cdot \chi(K) = (1-d) |H| .
    \]
    \begin{equation} \label{Equ:Euler-N}
       \Longrightarrow \ \ b_1(N;\F_p) = b_1 (X;\F_p)
          = 1 + b_2(X;\F_p) + (d-1) |H| . \qquad\qquad \ \
      \end{equation}
      
      By plugging~\eqref{Equ:b2-XN} into~\eqref{Equ:Euler-N}, we obtain
       \begin{align}
         b_1(N;\F_p)
          & \geq 1 + b_1(G;\F_p) \lambda^k_p(H) -
          d \cdot \Big(|H|- \sum_{0\leq j \leq k-1} \lambda^j_p(H) \Big) + (d-1) |H|
          \quad \notag\\
           & \geq 1 + b_1(G;\F_p)  \lambda^k_p(H) + d \sum_{0\leq j\leq k-1}
           \lambda^j_p(H) - |H| .  \notag
        \end{align}
   \end{proof}
   
  Next, we examine the special case of Theorem~\eqref{thm:Homology-Estimate-H}
  when $G\slash N = H \cong (\Z_p)^r$.
  First we recall a very useful result from~\cite{Jenning41}.
   \begin{lem}[Theorem 3.7 in~\cite{Jenning41}]\label{Lem:Jenning}
   For any $r\geq 1$, $k\geq 0$,
     $\lambda^k_p((\Z_p)^r)$ equals the coefficient of $x^k$ in the polynomial
      $(1+x+\cdots + x^{p-1})^r$.
  \end{lem}
  \n
   In other words, $\lambda^k_p((\Z_p)^r)$ equals the order $|\Omega^k_{p,r}|$ of the
   following set $\Omega^k_{p,r}$.
  \begin{equation} \label{Equ:Omega-def}
     \Omega^k_{p,r} =\{ (j_1,\cdots, j_r) \in \Z^r_{\geq 0}\ | \
   j_1+ \cdots + j_r = k,\, 0\leq j_i \leq p-1, \, i=1,\cdots, r \}.
    \end{equation}
 \nn
   Obviously, $|\Omega^k_{p,r}| >0$ if and only if $0\leq k \leq
   r(p-1)$, and we have
   $$\sum_{0\leq k\leq r(p-1)}  |\Omega^k_{p,r}| = p^r.   $$
    There is a bijection between the sets $\Omega^k_{p,r}$ and $\Omega^{r(p-1)-k}_{p,r}$
     where $ (j_1,\cdots, j_r) \in \Omega^k_{p,r} $ corresponds to $(p-1-j_1,\cdots,
     p-1-j_r) \in \Omega^{r(p-1)-k}_{p,r}$. So we have
     \begin{equation} \label{Equ:Omega-Sym}
        |\Omega^k_{p,r}| = |\Omega^{r(p-1)-k}_{p,r}|, \, 0\leq k \leq r(p-1). 
     \end{equation}   
  From
      $(1+x+\cdots + x^{p-1})^r = (1-x^p)^r (1-x)^{-r}$, we can compute $|\Omega^k_{p,r}|$ by:
   \begin{align} \label{Equ:Omega-Formula-1}
      |\Omega^k_{p,r}| &=  \sum_{0\leq j \leq \frac{k}{p}} (-1)^{k-(p-1)j} \binom{r}{j}
      \binom{-r}{k-pj} \notag\\
      & = \sum_{0\leq j \leq \frac{k}{p}} (-1)^j \binom{r}{j} \binom{k-pj+r-1}{k-pj}.
   \end{align}
   Obviously, $|\Omega^k_{p,r}| \geq |\Omega^k_{2,r}| = \binom{r}{k}$ for any prime $p$.
   More properties of these integers $|\Omega^k_{p,r}|$ are shown in Appendix-2.
 \nn
    By Theorem~\ref{thm:Homology-Estimate-H} and Lemma~\ref{Lem:Jenning}, 
    we obtain the following theorem immediately.\n
    
   \begin{thm} \label{thm:Homology-Estimate}
    Suppose a group $G$ admits a finite presentation of deficiency $d$. 
    Then for any prime $p$ 
    and any normal subgroup $N$ of $G$ with $G\slash N\cong (\Z_p)^r$,
    \begin{equation*}
       b_1(N;\F_p) \geq
       1 + b_1(G;\F_p) |\Omega^k_{p,r}| + d\, \sum^{k-1}_{i=0}
           |\Omega^i_{p,r}|- p^r \ \,\text{for each}\ 0\leq  k \leq r(p-1),
     \end{equation*}
   In particular, if $G\slash N \cong (\Z_2)^r$,
    \begin{equation*}
       b_1(N ;\F_2) \geq
         1 + b_1(G;\F_2) \binom{r}{k} + d \, \sum^{k-1}_{i=0}
         \binom{r}{i} - 2^r\ \, \text{for each}\   0\leq k \leq r.
     \end{equation*}
  \end{thm}
   
   \begin{rem}
   The condition $G\slash N \cong (\Z_p)^r$ implies that $r\leq b_1(G;\F_p)$. Indeed,
  let $\pi_N : G \rightarrow G\slash N \cong (\Z_p)^r$ be the quotient homomorphism.
  It is clear that 
  $$[G,G]G^p \subset \mathrm{ker}(\pi_N) =N.$$
   So from the fact that 
      $G\slash [G,G]G^p = H_1(G;\F_p)=(\Z_p)^{b_1(G;\F_p)}$, we obtain
   $$ (\Z_p)^r = G\slash N \cong \frac{G\slash [G,G]G^p}{N\slash [G,G]G^p}
   =\frac{(\Z_p)^{b_1(G;\F_p)}}{N\slash [G,G]G^p},\ \text{which implies}\
   r\leq b_1(G;\F_p). $$  
  \end{rem}
  \n   
  
 \begin{rem}
   The author was informed
    by a reviewer that Theorem~\ref{thm:Homology-Estimate} can also be derived 
    from a stronger result in Liam Wall's Ph.D thesis~\cite[Theorem 2.3.1]{Liam09}. The argument
    in~\cite{Liam09} uses some subtle choice of 
    group presentations taken from~\cite{Lack09} and Fox calculus. In addition,
    some other lower bounds of $b_1(N;\F_p)$ can be found in~\cite[Theorem 1.6]{Lack09}.
  \end{rem}   

        To give one a better idea of how powerful the lower bounds of
        $b_1(N;\F_p)$ are in Theorem~\ref{thm:Homology-Estimate},
        let us examine two special cases more carefully below, which
         leads to more straightforward lower bounds of $b_1(N;\F_p)$
            to fit our needs. \nn

  \begin{thm} \label{thm:Homology-Estimate-2}
     Let $G$ be a finitely presentable group with deficiency at least $1$.
     Then for any normal subgroup $N$ of $G$ with $G\slash N \cong (\Z_p)^r$, $r\geq 1$,
     $$  b_1(N;\F_p) \geq 2^{r-1}.$$
    In particular, if $ b_1(N;\F_p) = 2^{r-1}$, we must have
      $b_1(G;\F_p) = r = 1\ \text{or}\ 2$.
  \end{thm}
  \begin{proof}
    Set $d=1$ in Theorem~\ref{thm:Homology-Estimate}, we have for any
    $0 \leq k \leq r(p-1)$,
  \begin{align}
     b_1(N;\F_p)  &\geq  1 + b_1(G;\F_p) |\Omega^k_{p,r}| +  \sum^{k-1}_{i=0}
           |\Omega^i_{p,r}|- p^r  \notag \\
           & \geq 1 + (b_1(G;\F_p) -1)  |\Omega^k_{p,r}|
            + \sum^{k}_{i=0} |\Omega^i_{p,r}|- p^r  \notag \\
    &\geq  1 + (r-1)  |\Omega^k_{p,r}| -
  \sum^{r(p-1)}_{i=k+1} |\Omega^i_{p,r}|.
  \label{Equ:InEqu-0}
  \end{align}

   So to prove $b_1(N;\F_p)  \geq 2^{r-1}$, it is sufficient to
   show that there exists some $0 \leq k \leq r(p-1)$
   so that
   \[ \Pi^k_{p,r} := (r-1)  |\Omega^k_{p,r}| -
  \sum^{r(p-1)}_{i=k+1} |\Omega^i_{p,r}| \geq 2^{r-1}-1. \]
   \n

    \textbf{Claim 1:}  $\Pi^{r(p-1)-j}_{p,r} \geq \Pi^{r-j}_{2,r}$ for any prime $p$ and
    $0 \leq j \leq [\frac{r+1}{2}]$, and the equality holds only when $p=2$ or $j=1$.\nn

   The claim is trivial when $p=2$. When $p>2$, by~\eqref{Equ:Omega-Sym} we obtain
   \begin{align*}
     \Pi^{r(p-1)-j}_{p,r} &= (r-1) |\Omega^{j}_{p,r}| - \sum^{j-1}_{i=0}
      |\Omega^{i}_{p,r}|,\\
     \Pi^{r-j}_{2,r} &= (r-1) |\Omega^{j}_{2,r}| - \sum^{j-1}_{i=0}
       |\Omega^{i}_{2,r}|.
   \end{align*}
   Note that $|\Omega^{0}_{p,r}| =1$ for all prime $p$, so we have
   \begin{align*}
     \Pi^{r(p-1)-j}_{p,r} -  \Pi^{r-j}_{2,r}
   &= (r-1) \left(  |\Omega^{j}_{p,r}| -  |\Omega^{j}_{2,r}| \right)
      -  \sum^{j-1}_{i=1} \left( |\Omega^{i}_{p,r}| - |\Omega^{i}_{2,r}|
      \right).
   \end{align*}
   Moreover, Lemma~\ref{Lem:Omega-Prop-2} in Appendix-2 implies that for each $0\leq i < j \leq [\frac{r+1}{2}]
   \leq r$,
   $$ k_j := \frac{|\Omega^{j}_{p,r}|}{|\Omega^{j}_{2,r}|}
   \geq \frac{|\Omega^{i}_{p,r}|}{|\Omega^{i}_{2,r}|} :=k_i \geq 1
   \ \,  (\text{the equality holds only when}\ i=0,j=1).$$
   \begin{align*}
      \text{So}\ \  \Pi^{r(p-1)-j}_{p,r} -  \Pi^{r-j}_{2,r} &= (k_j-1) (r-1) |\Omega^{j}_{2,r}| -
      \sum^{j-1}_{i=1} (k_i-1) |\Omega^{i}_{2,r}|  \\
        &\geq (k_j-1) \Big(   (r-1) |\Omega^{j}_{2,r}| -
      \sum^{j-1}_{i=1} |\Omega^{i}_{2,r}|    \Big) \geq 0.
    \end{align*}
  The last ``$\geq$'' is because $|\Omega^{j}_{2,r}| =\binom{r}{j} \geq \binom{r}{i}=
    |\Omega^{i}_{2,r}|$
  when $1\leq i< j \leq [\frac{r+1}{2}]$.  \nn
 \n

     \textbf{Claim 2:} For a fixed $r\geq 1$, $\Pi^k_{2,r}$ reaches the maximum only
    at $k=[\frac{r+1}{2}]$. \nn

    Indeed, since $|\Omega^k_{2,r}| = \binom{r}{k}$,
     $\Pi^k_{2,r} - \Pi^{k-1}_{2,r}= r\binom{r}{k}
     -(r-1)\binom{r}{k-1}$. So

     \[ \Pi^k_{2,r} - \Pi^{k-1}_{2,r} >0 \ \Longleftrightarrow \
     \frac{r\binom{r}{k}}{ (r-1)\binom{r}{k-1}} =  \frac{r(r-k+1)}{(r-1)k} > 1
     \ \Longleftrightarrow \ k < \frac{r(r+1)}{2r-1} = \frac{r+1}{2-\frac{1}{r}}. \]
    So when $k \leq [\frac{r+1}{2}]$, $\Pi^k_{2,r} - \Pi^{k-1}_{2,r} > 0$.
    Similarly, we can show when $k \geq [\frac{r+1}{2}]+1$,
      $\Pi^k_{2,r} - \Pi^{k-1}_{2,r} < 0$. Therefore, $\Pi^k_{2,r}$ reaches the maximum at and only at
    $k=[\frac{r+1}{2}]$.
    The Claim 2 is proved.
    \nn

    By~\eqref{Equ:InEqu-0} and Claim 1, we have
    $$b_1(N;\F_p) \geq 1 + \Pi^{r(p-1)-[\frac{r}{2}]}_{p,r}
      \geq 1+ \Pi^{r-[\frac{r}{2}]}_{2,r} = 1+ \Pi^{[\frac{r+1}{2}]}_{2,r}.$$
    Next, we show
     \begin{equation} \label{Equ:Pi-Inequ}
     \Pi^{[\frac{r+1}{2}]}_{2,r} = (r-1)\binom{r}{[\frac{r+1}{2}]} - \sum^r_{i=[\frac{r+1}{2}] + 1}
      \binom{r}{i}
       \geq 2^{r-1}-1, \ \,\text{for}\ \forall\, r\geq 1.
     \end{equation}   \n

   \begin{itemize}
      \item  when $r=2t+1$, $t\geq 0$,  $[\frac{r+1}{2}] = t+1$,
    $$  \Pi^{[\frac{r+1}{2}]}_{2,r} =
   2t \binom{2t+1}{t+1} - \sum^{2t+1}_{j=t+2} \binom{2t+1}{j} = (2t+1) \binom{2t+1}{t+1} - 2^{2t}.$$
      So to prove $\Pi^{[\frac{r+1}{2}]}_{2,r} \geq 2^{r-1}-1 = 2^{2t}-1$, it is enough to
    prove
      \begin{equation} \label{Equ:InEqu-1}
         (2t+1) \binom{2t+1}{t+1} \geq 2^{2t+1} -1,\ \, \forall\,t\geq 0.
       \end{equation}
       \n

     \item  When $r=2t$, $t\geq 1$, $[\frac{r+1}{2}] = t$,
         $$   \Pi^{[\frac{r+1}{2}]}_{2,r} =
    (2t-1) \binom{2t}{t} - \sum^{2t}_{i=t+1} \binom{2t}{i}  = (2t-\frac{1}{2}) \binom{2t}{t} - 2^{2t-1}.$$
      So to prove $\Pi^{[\frac{r+1}{2}]}_{2,r} \geq 2^{r-1}-1 = 2^{2t-1}-1$, it is enough to
    prove
      \begin{equation} \label{Equ:InEqu-2}
         (2t-\frac{1}{2}) \binom{2t}{t} \geq 2^{2t} -1, \ \, \forall\, t\geq 1.
       \end{equation}
   \end{itemize}
   \n

 The proof of~\eqref{Equ:InEqu-1}
  and~\eqref{Equ:InEqu-2} is elementary and a bit tedious, so we leave it to the reader.
  Moreover, we find that\n
    \begin{itemize}
      \item  if the equality in~\eqref{Equ:InEqu-1} holds, $t$ must
      be $0$, hence $r=1$.\n

      \item  if the equality in~\eqref{Equ:InEqu-2} holds, $t$ must
      be $1$, hence $r=2$.
    \end{itemize}

  So when the equality $b_1(N;\F_p) = 2^{r-1}$ holds,
   the inequality in~\eqref{Equ:InEqu-0} implies that
   $b_1(G;\F_p) =r = 1$ or $2$. The theorem is proved.
  \end{proof}
  \n

  Notice that $\binom{r}{[\frac{r+1}{2}]}=\binom{r}{[\frac{r}{2}]}$.
  So~\eqref{Equ:Pi-Inequ} is equivalent to the following inequality
  \begin{equation}\label{Equ:Pi-Inequ-2}
      r\binom{r}{[\frac{r}{2}]} \geq 2^{r-1} +
      \sum^r_{i=[\frac{r+1}{2}]} \binom{r}{i} -1,\, \ \forall\, r\geq 1,
  \end{equation}
   where the equality holds only when $r=1$ or $2$.
  \nn
  \begin{thm} \label{thm:Homology-Estimate-3} 
     Suppose a group $G$ admits a balanced finite presentation. Then for any
     normal subgroup $N$ of $G$ with $G\slash N \cong (\Z_p)^r$, we have
     $$ b_1(N;\F_p) \geq  1+ r |\Omega^k_{p,r}| - p^r, \ 0\leq k \leq r(p-1).$$ 
    \end{thm}
  \begin{proof}
    Set $d=0$ in Theorem~\ref{thm:Homology-Estimate}, we get for any
    $0\leq k \leq r(p-1)$,
    $$  b_1(N;\F_p) \geq 1+ b_1(G;\F_p) |\Omega^k_{p,r}| - p^r \geq
      1+ r |\Omega^k_{p,r}| - p^r.$$
    \end{proof}
  \nn
 
 \section{Free $(\Z_p)^r$-actions on finite CW-complexes} \label{Sec3}
    Suppose there is a free $(\Z_p)^r$-action on a finite path-connected CW-complex $X$.
    Let $\xi: X \rightarrow K=X\slash (\Z_p)^r$ be the orbit map.
      Up to homotopy equivalence, we may assume that $K$ has
   a single $0$-cell $q_0$. Let the set of $1$-cells of $K$
     be $\{ \gamma_1,\cdots, \gamma_n \}$ and the set of $2$-cells of $K$
     be $\{ \beta_1,\cdots, \beta_m \}$.
    Let $x_0\in \xi^{-1}(q_0)$ be a basepoint of $X$.
    Since $K$ and $X$ are both path-connected, we have
    \begin{equation}
      b_1(K;\F_p) \geq r \geq 1,
    \end{equation}
    \begin{equation}
         b_1(K;\F_p) = b_1(\pi_1(K,q_0);\F_p),\ \
        b_1(X; \F_p) = b_1(\pi_1(X,x_0);\F_p).
    \end{equation}

        The fundamental group $\pi_1(K,q_0)$ of $K$
  has a natural presentation $\mathcal{P}^K$ defined by the $2$-skeleton
  of $K$. The deficiency of $\mathcal{P}^K$ is $n-m$.
  In addition,  since $\xi_* :
  \pi_1(X,x_0) \rightarrow \pi_1(K,q_0)$ is a monomorphism, we can identify $\pi_1(X,x_0)$ with its image
   $\xi_*(\pi_1(X,x_0))$. In the rest, we think of $\pi_1(X,x_0)$ as
  a normal subgroup of $\pi_1(K,q_0)$ and so have
  $\pi_1(K,q_0)\slash \pi_1(X,x_0) \cong (\Z_p)^r$.
   This allows us to use
  the theorems in the previous section to estimate $b_1(\pi_1(X,x_0);\F_p)$.
  Another important relation between $K$ and $X$ is that their Euler
  characteristics satisfy
  \begin{equation}
    \chi(X) = p^r \chi(K).
 \end{equation}

   Notice that if $\chi(K)\neq 0$, $\mathrm{hrk}(X;\F_p) \geq |\chi(X)| = p^r |\chi(K)| \geq p^r \geq
     2^r$.
     So the essential difficulty in the Halperin-Carlsson Conjecture for
     free $\Z_p$-torus actions lies in the case of $\chi(K)=\chi(X)=0$.
     We will use the above conventions and notations
      for $X$ and $K$ in all the proofs below.\vskip .3cm

 \noindent \emph{\textbf{Proof of Theorem~\ref{thm:Main1}.}}
      Since the dimension of $X$ and $K$ is $2$, the Euler characteristic of $K$ and $X$ are
       \begin{align}
        \chi(K) &= 1 - b_1(K;\F_p)  + b_2(K;\F_p) =  m-n+1, \label{Equ:Chi-K} \\
        \chi(X) &= 1- b_1(X; \F_p) + b_2(X; \F_p) = p^r\chi(K).
        \label{Equ:Chi-X}
       \end{align}
       \n

   \textbf{Case 1}:  When $\chi(K) \leq -1$, $\chi(X) \leq - p^r$. This implies
     $$ b_1(X; \F_p) \geq p^r  + b_2(X; \F_p) + 1.$$
      So
   $ \mathrm{hrk}(X; \F_p) =  1 + b_1(X; \F_p) + b_2(X; \F_p) \geq
   2 b_2(X; \F_p)  + p^r  +2 \geq p^r + 2 $.
   So in this case,
    $\mathrm{hrk}(X; \F_p)$ is strictly greater than
   $2^r$. \nn

   \textbf{Case 2}: When $\chi(K)=0$, we get $m=n-1$, $\chi(X)=0$ and
   $b_1(X; \F_p) =  1 + b_2(X; \F_p) $.
   So the presentation $\mathcal{P}^K$ of $\pi_1(K,q_0)$
   has deficiency $1$. Then Theorem~\ref{thm:Homology-Estimate-2}
   implies $ b_1(X; \F_p) = b_1(\pi_1(X,x_0);\F_p) \geq
   2^{r-1}$. So we get
    $$\mathrm{hrk}(X; \F_p) = 1+ b_1(X; \F_p) +
    b_2(X; \F_p) =2 b_1(X; \F_p) \geq 2^r.$$

  Moreover, when $\mathrm{hrk}(X; \F_p) =2^r$, we have $b_1(X; \F_p) = 2^{r-1}$.
    So Theorem~\ref{thm:Homology-Estimate-2}
   implies that $b_1(K;\F_p)=r=1$ or $2$, and we must have:\nn

                  \begin{itemize}
                    \item $b_1(K;\F_p) = r=1, b_2(K;\F_p)=0$, $b_1(X; \F_p) =
                    1$, $b_2(X; \F_p) = 0 $. Then
                     $$H_*(K;\F_p) \cong H_*(X; \F_p) \cong H_*(S^1;\F_p).$$ \n

                    \item $b_1(K;\F_p) = r= 2, b_2(K;\F_p)=1$, $b_1(X; \F_p) = 2$,
                    $b_2(X; \F_p) = 1$. Then
                    $$ H_*(K;\F_p) \cong H_*(X; \F_p) \cong H_*(S^1\times
                    S^1;\F_p).$$
                  \end{itemize}

  \nn

\textbf{Case 3}: When $\chi(K) \geq 1$,
   $$ \mathrm{hrk}(X; \F_p) =  \chi(X) + 2 b_1(X; \F_p) = p^r\chi(K) + 2 b_1(X; \F_p) \geq p^r \geq 2^r.$$
    In particular, if $\mathrm{hrk}(X; \F_p) = 2^r$, we must have
    $$p=2, \ \ \chi(K)=1, \ \ b_1(X;\F_2) =0.$$
    Then~\eqref{Equ:Chi-K} and~\eqref{Equ:Chi-X} implies $m=n$ and $b_2(X;\F_2) =
    2^r-1$.
   So the natural presentation $\mathcal{P}^K$ of $\pi_1(K,q_0)$ is balanced.
 Then by Theorem~\ref{thm:Homology-Estimate-3}, we get
    \begin{align*}
       b_1(X;\F_2) = b_1(\pi_1(X,x_0);\F_2)
         \geq 1+ r |\Omega^{[\frac{r}{2}]}_{2,r}| - 2^r & = 1+ r \binom{r}{[\frac{r}{2}]} -
         2^r \\
       (\text{by}~\eqref{Equ:Pi-Inequ-2})\ \,  &\geq   \sum^r_{i=[\frac{r+1}{2}]}
       \binom{r}{i} -2^{r-1} \geq 0.
    \end{align*}
 \begin{itemize}
   \item  The first inequality holds only when $b_1(K;\F_2) = r$;\nn
   \item The second inequality holds only when  $r= 1$ or $2$
   (see~\eqref{Equ:Pi-Inequ-2});\nn
   \item The third inequality holds only when $r$ is odd.\nn
 \end{itemize}

    So $b_1(X;\F_2) =0$ implies $b_1(K;\F_2) = r = 1$
     and $b_2(X;\F_2)=1$. In this case,
    $$ p=2,\ r=1, \ \ H_*(K;\F_2)\cong H_*(\R P^2;\F_2), \ \, H_*(X;\F_2)\cong H_*(S^2;\F_2).$$

    This finishes the proof of Theorem~\ref{thm:Main1}. \qed \\

 \noindent \emph{\textbf{Proof of Theorem~\ref{thm:Main2}.}}\
    First of all, we have
     \begin{equation} \label{Equ:Chi-X-2}
     \chi(X)=\sum_{i\geq 0} b_{2i}(X;\F_p) -  \sum_{i\geq 0}
     b_{2i+1}(X;\F_p) = p^r\chi(K).
     \end{equation}
    In addition, since we assume that the deficiency
     of $\pi_1(K,q_0)$ is at least $1$, so
       $$ b_1(X;\F_p) \geq 2^{r-1}\ \, \text{(by Theorem~\ref{thm:Homology-Estimate-2})}.$$\nn

     \textbf{Case 1}: When $\chi(K) \leq -1$,~\eqref{Equ:Chi-X-2}
     implies
     $$ \sum_{i\geq 0} b_{2i+1}(X;\F_p) = \sum_{i\geq 0} b_{2i}(X;\F_p) +  p^r \cdot |\chi(K)| \geq p^r +1 \geq 2^r+1.$$
     $$\Longrightarrow\ \mathrm{hrk}(X;\F_p) = \sum_{i\geq 0} b_{2i+1}(X;\F_p) + \sum_{i\geq 0} b_{2i}(X;\F_p)
        \geq (2^r +1) + 1 = 2^r+2 > 2^r.$$\nn

     \textbf{Case 2}: When $\chi(K)=0$, $\chi(X)=0$.  Then
    $$ \sum_{i\geq 0} b_{2i}(X;\F_p) =  \sum_{i\geq 0} b_{2i+1}(X;\F_p) \geq b_1(X;\F_p) \geq 2^{r-1}.$$
  Hence $\mathrm{hrk}(X;\F_p) \geq 2^{r-1} + 2^{r-1} = 2^r$.
   If $\mathrm{hrk}(X;\F_p) = 2^r$ in this case, we must have $b_1(X;\F_p) =
   2^{r-1}$, which implies $b_1(K;\F_p)=r=1$ or $2$ by
   Theorem~\ref{thm:Homology-Estimate-2}.\nn

     \textbf{Case 3}: When $\chi(K)\geq 1$,~\eqref{Equ:Chi-X-2}
     implies that $\chi(X)\geq p^r$. So we obtain
      \begin{align*}
        \mathrm{hrk}(X;\F_p) & = \chi(X) + 2\sum_{i\geq 0}
        b_{2i+1}(X;\F_p)\\
       & \geq p^r + 2 b_1(X;\F_p) \geq p^r + 2\cdot 2^{r-1} = p^r + 2^r > 2^r. \qquad  \qed
      \end{align*}
 \nn

   \noindent \emph{\textbf{Proof of Theorem~\ref{thm:Main3}.}}\
 Since we assume that the deficiency of $\pi_1(K,q_0)$ is at least $0$,
 $\pi_1(K,q_0)$ admits a balanced presentation. Then
  by Theorem~\ref{thm:Homology-Estimate-3}, 
    \begin{align}
       b_1(X;\F_2)  = b_1(\pi_1(X,x_0);\F_2) & \geq  
       1 + b_1(K;\F_2) |\Omega^{[\frac{r}{2}]}_{2,r}| - 2^r \label{Equ:Estim-M-1} \\
               & = 1+ b_1(K;\F_2) \binom{r}{[\frac{r}{2}]} - 2^r   \\
           & \geq 1+ r \binom{r}{[\frac{r}{2}]} - 2^r. \label{Equ:Estim-M-2}
     \end{align}
  Since when $r \leq 3$,~\cite[Theorem 1.1]{Puppe09} already implies $\mathrm{hrk}(X;\F_2) \geq
    2^r$. So we only need to deal with the $r\geq 4$ case.
     It is easy to verify the following combinatorial inequality (we leave the proof to the
     reader).
       \begin{equation}\label{Equ:InEqu-3}
    r \binom{r}{[\frac{r}{2}]} \geq 3\cdot 2^{r-1}  \ \ \text{for all}\ r\geq 4
   \end{equation}
   So when $r\geq 4$, $b_1(X;\F_2) \geq 2^{r-1} +1 $. Then by repeating the
   discussion of the three cases $\chi(K)\leq -1$, $\chi(K)=0$ and $\chi(K) \geq 1$ as
   the previous proof, we can easily see that $\mathrm{hrk}(X;\F_2) \geq 2^r + 2 > 2^r$ when $r\geq 4$.
   So for all $r\geq 1$, we have $\mathrm{hrk}(X;\F_2) \geq  2^r$. And
   if $\mathrm{hrk}(X;\F_2) = 2^r$, it is necessary that $r\leq 3$.
   \nn
   It remains to show $b_1(K;\F_2) = r$ when $\mathrm{hrk}(X;\F_2) = 2^r$.
  Indeed, if in this case $b_1(K;\F_2) > r$, then let us consider the regular covering space
   $$\widehat{\xi}: \widehat{X} \longrightarrow K$$
   where the fundamental group of $\widehat{X}$ is isomorphic (via
   $\widehat{\xi}_*$) to the kernel of
   the canonical group homomorphism $\psi: \pi_1(K,q_0) \rightarrow H_1(K;\Z) \rightarrow H_1(K;\F_2)$.
  So the deck transformation group of $\widehat{X}$ is isomorphic
  to $(\Z_2)^{b_1(K;\F_2)}$. On the other hand, since $\xi: X\rightarrow K$ is a regular
  $(\Z_2)^r$-covering, $\xi_*(\pi_1(X,x_0))$ should contain
  $\mathrm{ker}(\psi) = \widehat{\xi}_* (\pi_1(\widehat{X},\widehat{x}_0))$
  and hence
  $\xi_*(\pi_1(X,x_0)) \slash \widehat{\xi}_* (\pi_1(\widehat{X},\widehat{x}_0))
   \cong (\Z_2)^{b_1(K;\F_2)-r}$.
   So $\widehat{X}$ is a regular
  $(\Z_2)^{b_1(K;\F_2)-r}$-covering of $X$. Then there exists a
  sequence
  \[ \widehat{X}=X_{b_1(K;\F_2)-r} \longrightarrow \cdots \longrightarrow X_1 \longrightarrow X_0 = X  \]
  where each $X_i$ is a non-trivial double covering of $X_{i-1}$, $1\leq i \leq b_1(K;\F_2)-r$.
 Then by Lemma~\ref{Lem:Double_Covering} below, we can derive that
    $$\mathrm{hrk}(\widehat{X};\F_2) < 2^{b_1(K;\F_2)-r}\cdot  \mathrm{hrk}(X;\F_2) =
  2^{b_1(K;\F_2)-r}\cdot 2^r =  2^{b_1(K;\F_2)}.$$
  But since $\widehat{X}$ is a regular $(\Z_2)^{b_1(K;\F_2)}$-covering over $K$,
   we have already confirmed $\mathrm{hrk}(\widehat{X};\F_2) \geq 2^{b_1(K;\F_2)}$. This contradiction implies that
   if $\mathrm{hrk}(X;\F_2) = 2^r$, we must have $ b_1(K;\F_2) = r \leq 3$.
  \qed \nn

    \begin{lem} \label{Lem:Double_Covering}
           For any double covering $p: \widetilde{B} \rightarrow B$,
           $b_i(\widetilde{B};\F_2) \leq  2 \cdot b_i(B;\F_2)$ for all $i\geq 0$. So
          $ \mathrm{hrk}(\widetilde{B};\F_2) \leq 2 \cdot
          \mathrm{hrk}(B;\F_2)$. In particular, $\mathrm{hrk}(\widetilde{B};\F_2) = 2 \cdot
          \mathrm{hrk}(B;\F_2)$ if and only if $\widetilde{B}$ is the trivial double 
          covering of $B$.
       \end{lem}
       \begin{proof}
          The Gysin sequence of $p: \widetilde{B} \rightarrow B$ with $\F_2$-coefficients reads:
      \begin{equation} \label{Equ:Gysin-Seq}
           \cdots \longrightarrow  H^{i-1}(B;\F_2) \overset{\phi_{i-1}}{\longrightarrow}
            H^i(B;\F_2)
            \overset{p^*}{\longrightarrow} H^i(\widetilde{B};\F_2) \longrightarrow H^i(B;\F_2)
            \overset{\phi_i}{\longrightarrow}  \cdots
      \end{equation}
         where $\phi_i(\gamma) = \gamma \cup e,\ \forall\, \gamma \in H^i(B;\F_2)$ and
         $e\in H^1(B;\F_2)$ is the first Stiefel-Whitney
         class (Mod 2 Euler class) of $\widetilde{B}$. Then
          by the exactness of the Gysin sequence, we have:
       \begin{align*}
          b_i(\widetilde{B};\F_2) &=  \dim_{\F_2} H^i(\widetilde{B};\F_2) \\
           &= \dim_{\F_2} H^i(B;\F_2) - \dim_{\F_2}\text{Im}(\phi_{i-1}) + \dim_{\F_2}\text{ker}(\phi_i) \\
            &= 2 \cdot \dim_{\F_2} H^i(B;\F_2) -  \dim_{\F_2}\text{Im}(\phi_{i-1}) -
            \dim_{\F_2}\text{Im}(\phi_i) \\
            & \leq  2 \cdot \dim_{\F_2} H^i(B;\F_2) = 2b_i(B;\F_2).
      \end{align*}
          If $p: \widetilde{B} \rightarrow B$ is a non-trivial double
           covering, then $e \neq 0 \in  H^1(B;\F_2)$. This implies
           that $\mathrm{Im}(\phi_0)$ in~\eqref{Equ:Gysin-Seq} is not zero, and so $ b_0(\widetilde{B};\F_2) < 2 b_0(B;\F_2)$,
           $b_1(\widetilde{B};\F_2) < 2 b_1(B;\F_2)$. Then
            $ \mathrm{hrk}(\widetilde{B};\F_2) < 2 \cdot
            \mathrm{hrk}(B;\F_2)$. So if  $ \mathrm{hrk}(\widetilde{B};\F_2) = 2 \cdot
            \mathrm{hrk}(B;\F_2)$, $\widetilde{B}$ must be the trivial double covering of
            $B$.
       \end{proof}
       \nn

 \noindent \emph{\textbf{Proof of Corollary~\ref{Cor:3-manifolds}.}}\
    Suppose a closed connected $3$-manifold $M$ admits a free $(\Z_2)^r$-action.
    Then the orbit space $M\slash(\Z_2)^r = Q$ is also a
    closed connected $3$-manifold.
     Since the fundamental group of $Q$
    admits a balanced presentation from any Heegaard splitting of $Q$,
    Theorem~\ref{thm:Main3} tells us that $\mathrm{hrk}(M;\F_2) \geq 2^r$.\nn

   In particular, when $\mathrm{hrk}(M;\F_2) = 2^r$, it is necessary that $ b_1(Q;\F_2) = r \leq 3$.
     Since
     $ \mathrm{hrk}(M;\F_2) = 2 +b_1(M;\F_2) + b_2(M;\F_2) = 2^r$
     and Poincar\'e duality, we get
      $$b_1(M;\F_2) = b_2(M;\F_2) = 2^{r-1} -1.$$
      In
      addition, since $Q$ is a closed connected, we have
     $b_1(Q;\F_2) = b_2(Q;\F_2)$. So all the possible cases for $\mathrm{hrk}(M;\F_2) = 2^r$ are:\nn

    \begin{itemize}
      \item $b_1(Q;\F_2) =r=1$, so $b_1(M;\F_2) =b_2(M;\F_2) = 0$ and
         $$H_*(Q;\F_2) \cong H_*(\R P^3;\F_2), \ \  H_*(M;\F_2) \cong
        H_*(S^3;\F_2).$$

      \item $b_1(Q;\F_2) =r=2$, so $b_1(M;\F_2) =b_2(M;\F_2) =1$ and
         $$H_*(Q;\F_2) \cong H_*(S^1\times \R P^2;\F_2), \ \ H_*(M;\F_2)\cong H_*(S^1\times
         S^2;\F_2).$$

      \item $b_1(Q;\F_2) =r=3$, so $b_1(M;\F_2) = b_2(M;\F_2)=3$ and
       $$H_*(Q;\F_2) \cong H_*(M;\F_2) \cong H_*(S^1\times S^1\times
         S^1;\F_2).\qquad \qquad\qquad $$
      \end{itemize}

       So Theorem~\ref{thm:Main2} is proved. Note that all the 
       above three cases can be realized by free $(\Z_2)^r$-actions on some
       concrete $3$-manifolds.\qed\vskip .4cm
       
      \begin{rem}
        For a finite CW-complex $X$ which admits a free $(\Z_p)^r$-action,
         we may also consider
        the lower bounds estimates of $b_i(X;\F_p)$ for $i\geq 2$. 
        But similarly to $i=1$ case,
        such kind of lower bounds can be nontrivial only under 
       some extra conditions on $X$ (or the orbit space). 
       For example, 
       $(\Z_p)^r$ can act freely on a product of spheres 
       $X=S^{2l_1+1}\times \cdots \times S^{2l_r+1}$ 
         for any positive integer
          $l_1,\cdots, l_r$. But for a given $i\geq 1$,
           $b_i(X;\F_p)=0$ when $l_1,\cdots, l_r$ are all greater than $\frac{i}{2}$.
          So if we put no restrictions on the dimension of $X$, 
          there will be no nontrivial lower bounds of $b_i(X;\F_p)$ in general.
          This also suggests that in the general cases of the Halperin-Carlsson conjecture, 
          it is the sum $\sum_{i\geq 0} b_i(X;\F_p)$ but not any single
          term $b_i(X;\F_p)$ that should have a universal lower bound.
                \\
       \end{rem}

  \section{Appendix-1: The group ring $\F_p[H]$}
  
    For a finite group $H$, the group ring $\F_p[H]$ is a free module over $\F_p$ generated by
         all the elements of $H$. So we can think of the elements
      in $H$ forming a basis of $\F_p[H]$ over $\F_p$, denoted by
      $\{ \delta_h \, | \, h\in H\}$. Then any element $v$
      of $\F_p[H]$ can be written as
       $$v = \sum_{h\in H} l_h \delta_h, \, l_h\in \F_p.$$\n

    In the following, we use $\hat{0}$ to denote the zero element of $\F_p[H]$ to
         distinguish it from the scalar $0\in\F_p$.
       The product $*$ on the group ring $\F_p[H]$ is defined by
        \begin{equation}\label{equ:ast-1}
            \delta_{g} * \delta_h := \delta_{gh} , \   \delta_{g} * \hat{0} =  \hat{0}, 
            \ \ g, h\in H,\ k\in \F_p.
        \end{equation}
        \begin{align}\label{equ:ast-2}
           \sum_{g\in H} k_g \delta_g  *   \sum_{h\in H} l_h \delta_h
           & := \sum_{g\in H} \sum_{h\in H} k_g l_h  (\delta_{g} *
           \delta_h) =  \sum_{g\in H} \sum_{h\in H} k_g l_h
                \delta_{gh}. 
        \end{align}

   The product $*$ is commutative if and only if $H$ itself is commutative.\nn

 Notice that $\delta_e * \delta_h = \delta_h*\delta_e = \delta_h $ for all $h\in H$.
  So $\delta_e$ can be identified
  with the scalar $1\in \F_p$, and the field $\F_p$ can be embedded in
 $\F_p[H]$ via the map
 \begin{align*}
  \iota: \F_p \, &\longrightarrow \, \F_p[H] \\
    k   \, &\longmapsto\, k\delta_e
 \end{align*}

 There also exists a canonical ring homomorphism going the other way, called the \emph{augmentation}.
  It is the map $\eta: \F_p[H] \rightarrow \F_p$, defined by
  \[ \eta \big( \sum_{h\in H} l_h \delta_h \big) = \sum_{h\in H} l_h \in \F_p. \]
  The kernel of $\eta$ is called the \emph{augmentation ideal} of $\F_p[H]$, denoted by
  $\Delta_{\F_p}(H)$. Indeed, $\Delta_{\F_p}(H)$ is a free $\F_p$-module generated by the
 set $\{ -\delta_e + \delta_h  \, ; \, h\in H \}$ and
      $$\F_p[H] = \Delta_{\F_p}(H) \oplus \F_p .$$
       The reader is referred to~\cite{MilisSehg02} and~\cite{Passi79}
     for more information of group rings and their augmentation ideals.
       \nn

       \begin{rem}
   By identifying $1\in \F_p$ with $\delta_e$,
   we will write $ -\delta_e + \delta_h \in  \Delta_{\F_p}(H)$ instead
    of $-1 + \delta_h$ in this paper.\nn
\end{rem}

 There is a natural filtration of $\F_p[H]$ as shown below.
   \begin{equation} \label{Equ:Filtration-GroupRing}
    \F_p[H]   \supset \Delta_{\F_p}(H)  \supset  \Delta^2_{\F_p}(H) \supset\cdots \supset
      \Delta^k_{\F_p}(H) \supset \Delta^{k+1}_{\F_p}(H) \supset \cdots
   \end{equation}
   By abuse of notation, let $ \Delta^0_{\F_p}(H) := \F_p[H]$. In addition, define
    \begin{equation} \label{Equ:lambda-k-H}
     \lambda^k_p(H) := \dim_{\F_p}\Delta^k_{\F_p}(H) - \dim_{\F_p}
     \Delta^{k+1}_{\F_p}(H) =  \dim_{\F_p}\Delta^k_{\F_p}(H)\slash \Delta^{k+1}_{\F_p}(H).
    \end{equation}
    So we have
    \begin{equation} \label{Equ:lambda-k-H-2}
      \dim_{\F_p}\Delta^k_{\F_p}(H) = \sum_{j\geq k}
      \lambda^j_p(H) = |H| - \sum_{0\leq j \leq k-1} \lambda^j_p(H).
    \end{equation}
   Since $H$ is a finite group, the filtration~\eqref{Equ:Filtration-GroupRing} becomes stable
   after finitely many steps. In particular, $\Delta^k_{\F_p}(H)  = \{
   \hat{0}\}$ for some $k$
   if and only if $\Delta_{\F_p}(H)$ is nilpotent.
  A well known fact (\cite[Theorem 9]{Conn63})
  says that for any nontrivial group $G$, the augmentation ideal of the group ring $R[G]$
   over a unital ring $R$ is nilpotent if and only if
  $G$ is a finite $p$-group and $p$ is nilpotent in $R$.
 \\

  \section{Appendix-2: Properties of $|\Omega^m_{p,r}|$ } \label{Appendix 2}
  In the following, we investigate the properties of the family of integers $|\Omega^m_{p,r}|$  defined by~\eqref{Equ:Omega-def}
   and compare them with binomial coefficients.
  These properties of $|\Omega^m_{p,r}|$ are useful for the proof of
   Theorem~\ref{thm:Main1} and Theorem~\ref{thm:Main2}.\nn

   First of all, $|\Omega^0_{p,r}| = 1$, $|\Omega^1_{p,r}| = r$ for any prime $p$, and
   \begin{equation} \label{Equ:Omega-p-1}
      |\Omega^m_{p,1}| = \left\{
                           \begin{array}{ll}
                             1, & \hbox{\text{$0\leq m \leq p-1$};} \\
                             0, & \hbox{\text{otherwise}.}
                           \end{array}
                         \right.
   \end{equation}
     Moreover, we have the following recursive relation
     \begin{equation}\label{Equ:Recur-Relation}
      |\Omega^m_{p,r}| =   |\Omega^m_{p,r-1}| +  |\Omega^{m-1}_{p,r-1}| + \cdots +
        |\Omega^{m-(p-1)}_{p,r-1}|, \ \, \forall\, m\in \Z.
     \end{equation}

    An easy way to prove this recursive relation is to understand $|\Omega^m_{p,r}|$ as the
   number of different ways of putting $m$ indistinguishable balls into $r$
     different bags where each bag can hold at most $p-1$ balls.
     If there are exactly $i$ balls in the first bag, then the remaining
      $m-i$ balls must be put in
     other $r-1$ bags, which gives the term $|\Omega^{m-i}_{p,r-1}|$
     on the right hand side of~\eqref{Equ:Recur-Relation}. \nn
  
   To see how fast $|\Omega^{m}_{p,r}|$ increases with respect to $m$, let us
   compare $|\Omega^{m}_{p,r}|$ with binomial coefficients. Since the formula~\eqref{Equ:Omega-Formula-1}
    of $|\Omega^{m}_{p,r}|$
    is an alternating sum, it is not so convenient for our purpose. So we give another way to compute
    $|\Omega^{m}_{p,r}|$ as follows.
    Let $\Theta^m_{p,r}$ be the set of all partitions of $m$ into $r$ nonnegative integers which are
    no more than $p-1$.
    Each element of $\Theta^m_{p,r}$ can be uniquely represented by an integral vector
    \begin{equation} \label{Equ:alpha}
       \alpha = \big(\,\overset{l_0}{\overbrace{0,\cdots, 0}},\,
       \overset{l_1}{\overbrace{n_1,\cdots, n_1}},\, \overset{l_2}{\overbrace{n_2,\cdots, n_2}},
    \cdots, \overset{l_s}{\overbrace{n_s,\cdots, n_s}}\,\big) \in \Z^r_{\geq 0},
    \end{equation}
    where  $l_0+l_1 +\cdots + l_s = r$, $0 < n_1 < n_2  < \cdots < n_s \leq p-1$
    and $l_1n_1 + \cdots + l_sn_s = m$.
    Note that $l_0$ could be $0$.
    For convenience, define
      \begin{equation} \label{Equ:L-alpha}
          L_{\alpha} :=  l_0 !\, l_1! \, \cdots\, l_s!.
      \end{equation}
   There is a natural map $\varrho: \Omega^{m}_{p,r} \rightarrow
   \Theta^{m}_{p,r}$ which sends $(j_1,\cdots, j_r)\in  \Omega^{m}_{p,r}$ to the
   corresponding partition of $m$ in $ \Theta^{m}_{p,r}$. It is clear that
   $\varrho$ is surjective and we have
   $$|\varrho^{-1}(\alpha)| =  \frac{r!}{L_{\alpha}}, \ \, \alpha\in  \Theta^{m}_{p,r}. $$

     \begin{equation} \label{Equ:Omega-Formula-2}
    \Longrightarrow \  |\Omega^{m}_{p,r}|  =
      \sum_{\alpha \in \Theta^m_{p,r}} \frac{r!}{L_{\alpha}}.
      \qquad\qquad
    \end{equation}

   \nn

    \begin{lem} \label{Lem:Omega-Prop-2}
   For any prime $p$ and any $1\leq m\leq r$,
    \[
   \frac{|\Omega^{m}_{p,r}|}{|\Omega^{m-1}_{p,r}|}
          \geq   \frac{|\Omega^{m}_{2,r}|}{|\Omega^{m-1}_{2,r}|} = \frac{\binom{r}{m}}{\binom{r}{m-1}} =
  \frac{r-m+1}{m}.
          \]
     Moreover, the equality holds only when $p=2$ or $m=1$.
    \end{lem}
 \begin{proof}
   First of all, when $1\leq m \leq r$, any element $\beta\in \Theta^{m-1}_{p,r}$
   must have the form
    \[ \beta = \big(\,\overset{l_0}{\overbrace{0,\cdots, 0}},\,
       \overset{l_1}{\overbrace{n_1,\cdots, n_1}},\, \overset{l_2}{\overbrace{n_2,\cdots, n_2}},
    \cdots, \overset{l_s}{\overbrace{n_s,\cdots, n_s}}\,\big)\ \text{with}\ l_0\geq 1.  \]
    This is because $l_1n_1 + \cdots + l_sn_s = m-1$ implies
     $l_1+ \cdots + l_s \leq m-1$. Then from $l_0+l_1 +\cdots + l_s = r$, we
     get
       \begin{equation} \label{Equ:l_0}
          l_0 \geq r- (m-1) = r-m+1 \geq 1.
       \end{equation}
   This fact allows us to embed $\Theta^{m-1}_{p,r}$ into $\Theta^m_{p,r}$
   by sending $\beta \in \Theta^{m-1}_{p,r}$ to
    $$\beta^+ \in \Theta^m_{p,r}:=
    \big(\,\overset{l_0-1}{\overbrace{0,\cdots, 0}},\, 1,\,
       \overset{l_1}{\overbrace{n_1,\cdots, n_1}},\, \overset{l_2}{\overbrace{n_2,\cdots, n_2}},
    \cdots, \overset{l_s}{\overbrace{n_s,\cdots, n_s}}\,\big).  $$
   This map is obviously injective.
    By the definition~\eqref{Equ:L-alpha}, we have
   $$L_{\beta^+} = \left\{
                    \begin{array}{ll}
                    (l_0-1)!\, (l_1+1)!\, l_2!\,\cdots l_s!, & \hbox{\text{$n_1=1$};} \\
                     (l_0-1)!\, l_1!\, l_2!\,\cdots l_s!, & \hbox{\text{$n_1 > 1$}.}
                    \end{array}
                  \right.
    $$
  \[\Longrightarrow\  L_{\beta^+} = \left\{
                    \begin{array}{ll}
                      \frac{l_1+1}{l_0} L_{\beta}, & \hbox{\text{$n_1=1$};} \\
                   \frac{1}{l_0}  L_{\beta}, & \hbox{\text{$n_1 > 1$}.}
                    \end{array}
                  \right. \ \ \frac{r!}{L_{\beta^+}} = \left\{
                                      \begin{array}{ll}
                                         \frac{l_0}{l_1+1}\cdot \frac{r!}{L_{\beta}},
                          & \hbox{\text{$n_1=1$};} \\
      l_0\cdot \frac{r!}{L_{\beta}},
                           & \hbox{\text{$n_1>1$}.}
                                      \end{array}
                                    \right.
  \]

  \ \\
 Then since $l_0 \geq r- m +1 $ and $0\leq l_1\leq  m-1$, we have
  $l_0 \geq \frac{l_0}{l_1 + 1} \geq \frac{r-m+1}{m}$.
   So
 \[   \frac{r!}{L_{\beta^+}} \geq  \frac{r-m+1}{m}\cdot \frac{r!}{L_{\beta}}, \ \,
  \beta \in \Theta^{m-1}_{p,r}, 1\leq m \leq r. \]
 Then by~\eqref{Equ:Omega-Formula-2}, for any $1\leq m\leq r$, we get
 \begin{align*}
  |\Omega^{m}_{p,r}| =
    \sum_{\alpha \in \Theta^m_{p,r}}  \frac{r!}{L_{\alpha}}
      & \geq  \sum_{\beta \in \Theta^{m-1}_{p,r}}
      \frac{r!}{L_{\beta^+}}  \geq  \sum_{\beta \in \Theta^{m-1}_{p,r}}
      \frac{r-m+1}{m} \cdot  \frac{r!}{L_{\beta}} = \frac{r-m+1}{m}\cdot |\Omega^{m-1}_{p,r}|.
  \end{align*}
  The first ``$\geq$'' is because $\{ \beta^+ \, |\, \beta\in \Theta^{m-1}_{p,r}
  \}$ is a subset of $\Theta^m_{p,r}$.
  When $p=2$, the lemma is trivial.
  When $p >2$, note that for $r\geq m\geq 2$, the partition 
  $(m) \in \Theta^m_{p,r}$ cannot
   be written as $\beta^+$ for any $\beta\in \Theta^{m-1}_{p,r}$,
   which implies $\{ \beta^+ \, |\, \beta\in \Theta^{m-1}_{p,r}
  \}$ is a proper subset of $\Theta^m_{p,r}$. So
   $ |\Omega^{m}_{p,r}| > \frac{r-m+1}{m}\cdot |\Omega^{m-1}_{p,r}|$ for $p>2$ and
    $r\geq m\geq 2$.
  And when $m=1$, we clearly have $ |\Omega^{m}_{p,r}| =
  \frac{r-m+1}{m} \cdot |\Omega^{m-1}_{p,r}|=r$. So the lemma is
  proved.
  \end{proof}

\end{document}